\documentclass[oneside]{article}

\usepackage{tikz}
\definecolor{bl}{RGB}{20,20,150}
\usepackage{graphicx} 
\usepackage{verbatim}
\usepackage{amsmath}
\usepackage{amssymb}
\usepackage{amsthm}
\usepackage{dsfont}
\usepackage{array}
\usepackage[hyphens]{url}
\usepackage{hyperref}

\usepackage[margin=1.5in]{geometry}

\newtheorem{theorem}{Theorem}
\newtheorem{lemma}[theorem]{Lemma}
\newtheorem{proposition}[theorem]{Proposition}
\newtheorem{corollary}[theorem]{Corollary}

\theoremstyle{definition}

\newtheorem{example}[theorem]{Example}
\newtheorem{remark}[theorem]{Remark}

\newcommand{\Ecal}{\mathcal{E}}

\newcommand{\Mcal}{\mathcal{M}}

\newcommand{\Vcal}{\mathcal{V}}

\newcommand{\bfB}{{\bf B}}

\newcommand{\XX}{\mathbb{X}}
\newcommand{\YY}{\mathbb{Y}}
\newcommand{\RR}{\mathbb{R}}

\newcommand{\nv}{v}
\newcommand{\nh}{h}

\newcommand{\RBM}{\operatorname{RBM}}
\newcommand{\supp}{\operatorname{supp}}

\DeclareMathOperator*{\bigtimes}{\textnormal{\Large $\times$}} % Cartesian Product

\newcommand{\keywords}[1]{\smallskip\newline\noindent{\bf Keywords:} #1}
\graphicspath{ {./Figures/}{./Computations/} }

\title{\sc Hierarchical Models as Marginals of Hierarchical Models}
\author{
{\small \bf Guido Mont\'ufar} \\ \small Max Planck Institute for \\ \small Mathematics in the Sciences\\ \small montufar@mis.mpg.de
\and 
{\small \bf Johannes Rauh} \\ \small Department of Mathematics and Statistics \\ \small York University \\ \small jarauh@yorku.ca
 }
\date{\small\today}

\begin{document}
\maketitle
\thispagestyle{empty}

\begin{abstract}%
We investigate the representation of hierarchical models in terms of marginals of other hierarchical models with smaller interactions. We focus on binary variables and marginals of pairwise interaction models whose hidden variables are conditionally independent given the visible variables. In this case the problem is equivalent to the representation of linear subspaces of polynomials by feedforward neural networks with soft-plus computational units. We show that every hidden variable can freely model multiple interactions among the visible variables, which allows us to generalize and improve previous results. In particular, we show that a restricted Boltzmann machine with less than $[ 2(\log(v)+1) / (v+1) ] 2^v-1$ hidden binary variables can approximate every distribution of $v$ visible binary variables arbitrarily well,  compared to $2^{\nv-1}-1$ from the best previously known result.  
\end{abstract}

\section{Introduction}

Consider a finite set $V$ of random variables. 
A hierarchical log-linear model is a set of joint probability distributions that can be written as products of interaction potentials, as $p(x) = \prod_{\Lambda} \psi_\Lambda(x)$, where $\psi_\Lambda(x) = \psi_\Lambda(x_\Lambda)$ only depends on the subset $\Lambda$ of variables and where the product runs over a fixed family of such subsets. 
By introducing hidden variables, it is possible to express the same probability distributions in terms of potentials which involve only small sets of variables, as $p(x) = \sum_{y} \prod_{\lambda}\psi_\lambda(x,y)$, with small sets~$\lambda$. 
Using small interactions is a central idea in the context of connectionistic models, where the sets $\lambda$ are often restricted to have cardinality two. 
Due to the simplicity of their local characteristics, these models are particularly well suited for Gibbs sampling~\cite{4767596}. 
The representation, or explanation, of complex interactions among observed variables in terms of hidden variables is also related to the study of common ancestors~\cite{e17042304}. 

We are interested in sufficient and necessary conditions on the number of hidden variables, their values, and the interaction structures, under which the visible marginals are flexible enough to represent any distribution from a given hierarchical model. 
Many problems can be formulated as special cases of this general problem. 
For example, the problem of calculating 
%the smallest mixtures of product distributions that can represent a given hierarchical model~\cite{montufar2013mixture}, the smallest number of hidden units of a restricted Boltzmann machine~\cite{LeRoux:2008:RPR:1374176.1374187,montufar2011refinements} 
the smallest number of layers of variables that a deep Boltzmann machine needs in order to represent any probability distribution~\cite{montufar2015deep}. 

In this article, we focus on the case that all variables are binary. 
For the hierarchical models with hidden variables, we restrict our attention to models involving only pairwise interactions and whose hidden variables are conditionally independent given the visible variables (no direct interactions between the hidden variables). 
A prominent example of this type of models is the restricted Boltzmann machine, which has full bipartite interactions between the visible and hidden variables. 
The representational power of restricted Boltzmann machines has been studied assiduously; see, e.g.,~\cite{NIPS1991_535,Younes1996109,LeRoux:2008:RPR:1374176.1374187,montufar2011refinements}. 
The free energy function of such a model is a sum of soft-plus computational units $x\mapsto \log(1+\exp(\sum_{i\in V} w_i x_i + c))$. 
On the other hand, the energy function of a fully observable hierarchical model with binary variables is a polynomial, with monomials corresponding to pure interactions. 
Since any function of binary variables can be expressed as a polynomial, 
the task is then to characterize the polynomials computable by soft-plus units. 

Younes~\cite{Younes1996109} showed that a hierarchical model with $N$ binary variables and a total of $M$ pure higher order interactions (among three or more variables) can be represented as the visible marginal of a pairwise interaction model with $M$ hidden binary variables. 
In Younes' construction, each pure interaction is modeled by one hidden binary variable that interacts pairwise with each of the involved visible variables. 
In fact, he shows that this replacement can be accomplished without increasing the number of model parameters, by imposing linear constraints on the coupling strengths of the hidden variable. 
In this work we investigate ways of squeezing more degrees of freedom out of each hidden variable. An indication that this should be possible is the fact that the full interaction model, for which $M = 2^N - \binom{N}{2} - N - 1$, can be modeled by a pairwise interaction model with $2^{N-1}-1$ hidden variables~\cite{montufar2011refinements}. 
Indeed, by controlling groups of polynomial coefficients at the time, we show that in general less than $M$ hidden variables are sufficient. 

A special case of hierarchical models with hidden variables are mixtures of hierarchical models. 
The smallest mixtures of hierarchical models that contain other hierarchical models have been studied in~\cite{montufar2013mixture}. 
The approach followed there is different and complementary to our analysis of soft-plus polynomials. 
For the necessary conditions, the idea there is to compare the possible support sets of the limit distributions of both models. 
For the sufficient conditions, the idea is to find a small $S$-set covering of the set of elementary events. 
An $S$-set of a probability model is a set of elementary events such that every distribution supported in that set is a limit distribution from the model. 

Another type of hierarchical models with hidden variables are tree models. The geometry of binary tree models has been studied in~\cite{Piotr:tree} in terms of moments and cumulants. That analysis bears some relation to ours in that it also elaborates on M\"obius inversions. 

\medskip
This paper is organized as follows. 
Section~\ref{section:preliminaries} introduces hierarchical models and formalizes our problem in the light of previous results. 
Section~\ref{section:softplus} pursues a characterization of the polynomials that can be represented by soft-plus units. 
Section~\ref{section:RBMs} applies this characterization to study the representation of hierarchical models in terms of pairwise interaction models with hidden variables. This section addresses principally restricted Boltzmann machines. 
Section~\ref{section:conclusions} offers our conclusions and outlook.  

\section{Preliminaries}
\label{section:preliminaries}

This section introduces hierarchical models, with and without hidden variables, formalizes the problem that we address in this paper, and presents motivating prior results. 

\subsection{Hierarchical Models}

Consider a finite set $V$ of variables with finitely many joint states $x =(x_i)_{i\in V}\in \XX =\bigtimes_{i\in V}\XX_i$. 
We write $\nv = |V|$ for the cardinality of $V$. 
For a given set $S\subseteq 2^V$ of subsets of $V$ let 
\begin{equation*}
\Vcal_{\XX,S} := \left\{ g(x) = \sum_{\Lambda\in S} g_\Lambda(x) \colon  g_\Lambda(x) = g_\Lambda(x_\Lambda)\right\}. 
\end{equation*} 
This is the linear subspace of $\RR^{\XX}$ spanned by functions $g_\Lambda$ that only depend on sets of variables $\Lambda\in S$. 
The hierarchical model of probability distributions on $\XX$ with interactions $S$ is the set 
\begin{equation}
\Ecal_{\XX,S} := \left\{ p(x) = \frac{1}{Z(g)}\exp(g(x)) \colon g\in\Vcal_{\XX,S} \right\}, 
\label{eq:hierarchical}
\end{equation} 
where $Z(g) = \sum_{x'\in\XX}\exp(g(x'))$ is a normalizing factor. 
We call
\begin{equation}
E(x) = g(x) = \sum_{\Lambda\in S} g_\Lambda(x)
\label{eq:1}
\end{equation}
the \emph{energy function} of the corresponding probability distribution.

For convenience, in all what follows we assume that $S$ is a simplicial complex, meaning that $A\in S$ implies $B\in S$ for all $B\subseteq A$. 
Furthermore, we assume that the union of elements of $S$ equals $V$. 
In the case of binary variables, $\XX_i=\{0,1\}$ for all $i\in V$, the energy can be written as a polynomial, as
\begin{equation*}
E(x) = \sum_{\Lambda\in S} J_\Lambda \prod_{i\in\Lambda}x_i. 
\end{equation*}
Here, $J_\Lambda\in \RR$, $\Lambda\in S$, are the interaction weights that parametrize the model.

\subsection{Hierarchical Models with Hidden Variables}

Consider an additional set $H$ of variables with finitely many joint states $y = (y_j)_{j\in H} \in \YY=\bigtimes_{j\in H} \YY_j$. 
We write $\nh = |H|$ for the cardinality of $H$. 
For a simplicial complex $T\subseteq 2^{V\cup H}$, let $\Vcal_{\XX\times\YY,T} \subseteq\RR^{\XX \times \YY}$ be the linear subspace of functions of the form 
$g(x,y) = \sum_{\lambda\in T} g_\lambda(x,y)$,  $g_\lambda(x,y) = g_\lambda((x,y)_\lambda)$. 
The marginal on $\XX$ of the hierarchical model $\Ecal_{\XX\times\YY, T}$ is the set
\begin{equation}
\Mcal_{\XX\times\YY, T} :=\left\{ p(x) = \frac{1}{Z(g)}\sum_{y\in\YY } \exp(g(x,y)) \colon g\in\Vcal_{\XX\times\YY,T}\right\}, 
	\label{eq:marginal}
\end{equation}
where $Z(g)= \sum_{x'\in\XX,y'\in\YY}\exp(g(x',y'))$ is again a normalizing factor. 
The free energy of a probability distribution from $\Mcal_{\XX\times\YY,T}$ is given by 
\begin{equation}
F(x) 
= \log \sum_{y\in \YY} \exp\Big(\sum_{\lambda\in T} g_\lambda(x,y) \Big). 
\label{eq:2}
\end{equation}
Here and throughout ``$\log$'' denotes the natural logarithm. 

In the case of binary visible variables, $\XX_i=\{0,1\}$ for all $i\in V$, 
%and arbitrary interactions, 
the free energy~\eqref{eq:2} can be written as a polynomial, as
\begin{equation*}
F(x) = \sum_{B \subseteq V} K_B \prod_{i\in B} x_i , 
\end{equation*}
where the coefficients can be computed from M\"obius' inversion formula as
\begin{equation}
K_B = \sum_{C\subseteq B} (-1)^{|B\setminus C|} \log \sum_{y\in\YY} \exp \Big(\sum_{\lambda\in T }  g_\lambda((1_{C},0_{V\setminus C}),y) \Big), \quad B\in 2^V. 
\label{eq:moebius}
\end{equation}
Here $(1_C,0_{V\setminus C})\in\{0,1\}^V$ is the vector with value $1$ in the entries $i\in C$ and value $0$ in the entries $i\not\in C$. 

If there are no direct interactions between hidden variables, i.e.,~if $|\lambda\cap H|\le 1$ for all $\lambda\in T$, then the sum over $y$ factorizes and the free energy~\eqref{eq:2} can be written as
\begin{equation}
  \label{eq:independent-hidden-variables}
  F(x) = \sum_{\lambda:\lambda\cap H=\emptyset} g_{\lambda}(x)
  + \sum_{j\in H}\log \sum_{y_{j}\in\YY_{j}}\exp\Big(\sum_{\lambda\in T: j\in\lambda} g_{\lambda}(x,y_{j})\Big).
\end{equation}
Particularly interesting are the models with full bipartite interactions between the set of visible variables and the set of hidden variables, i.e., models with $T=\{\lambda\subseteq V \cup H \colon |\lambda\cap V| \leq 1, |\lambda\cap H|\leq 1\}$, called restricted Boltzmann machines (with discrete variables). 

\begin{comment}
If all variables are binary and there are no direct interactions between hidden variables and between hidden and visible variables there are only pairwise interactions, then the free energy~\eqref{eq:2} can be written as 
\begin{equation}
  \label{eq:allbinary-independent-hidden-variables}
  F(x) = \sum_{\lambda\in 2^V} b_\lambda \prod_{i\in\lambda}x_i
  + \sum_{j\in H}\log \Big( 1 + \exp(\sum_{i\in V} w_{j,i}x_i + c_j)\Big). 
\end{equation}
Here, $b_\lambda$, $\lambda\in 2^V$, are interaction weights between visible variables, $w_{j,i}$, $j\in H$, $i\in V$, are pair interaction weights between hidden and visible variables, and $c_j$, $j\in H$, are biases of the hidden variables.  
\end{comment}

\subsection{Problem and Previous Results}
\label{sec:problem}

In general the marginal of a hierarchical model is not a hierarchical model. 
%\todo{Is there an intuitive explanation for this?} 
However, one may ask which hierarchical models are contained in the marginal of another hierarchical model. 

To represent a hierarchical model in terms of the marginal of another hierarchical model, 
we need to represent~\eqref{eq:hierarchical} in terms of~\eqref{eq:marginal}. 
Equivalently, we need to represent all possible energy functions in terms of free energies. 
Given a set of visible variables $V$ and a simplicial complex $S\subseteq 2^V$, 
what conditions on the set of hidden variables $H$ and the simplicial complex $T\subseteq 2^{V\cup H}$ are sufficient and necessary in order for any function $E$ of the form~\eqref{eq:1} to be representable in terms of some function $F$ of the form~\eqref{eq:2}? 
We would like to arrive at a result that generalizes the following. 

\begin{itemize}
	\item 
	A restricted Boltzmann machine with $\nh$ hidden binary variables can approximate any probability distribution from a binary hierarchical model $\Ecal_{S}$ with $|\{\Lambda\in S\colon |\Lambda|>1\}| \leq \nh$ arbitrarily well~\cite{Younes1996109}. 
	\item 
	The restricted Boltzmann machine with $\nh = 2^{\nv-1}-1$ hidden binary variables can approximate any probability distribution of $\nv$ binary variables arbitrarily well~\cite{montufar2011refinements}. 
%	\item Every probability distribution on $\{0,1\}^V$ can be approximated arbitrarily well by a mixture of $k$ fully factorizing probability distributions if and only if $k\geq 2^{\nv-1}$. See~\cite{montufar2013mixture}. 
\end{itemize}
Our Theorem~\ref{theorem2a} in Section~\ref{section:RBMs} improves and generalizes these statements. 
The basis of this result are soft-plus polynomials, which we discuss in the following section. 
%The third item is an example of a tight bound, providing sufficient and necessary conditions. The set of mixtures of $k$ fully factorizing probability distributions corresponds to the hierarchical model with one $k$-valued hidden variable that interacts pairwise with each visible variable. 

\section{Soft-plus Polynomials}
\label{section:softplus}

Consider a function of the form 
\begin{equation}
\phi\colon \{0,1\}^V\to \RR;\;  x\mapsto \log(1 + \exp(w^\top x + c)),
\label{eq:softplus}
\end{equation} 
parametrized by $w=(w_i)_{i\in V}\in \RR^V$ and $c\in \RR$. 
We regard $\phi$ as a \emph{soft-plus computational unit}, which integrates each input vector $x\in\{0,1\}^V$ into a scalar via $x\mapsto w^\top x +c$ and applies the soft-plus non-linearity $f\colon \mathbb{R}\to\mathbb{R}_+;\;s\mapsto \log(1+ \exp(s))$. 
See Figure~\ref{figure:softplus} for an illustration of this function. 
In view of Equation~\eqref{eq:independent-hidden-variables}, the function $\phi$ corresponds to the free energy added by one hidden binary variable interacting pairwise with $V$ visible binary variables. The parameters $w_i$, $i\in V$, correspond to the pair interaction weights and $c$ to the bias of the hidden variable. 

\begin{figure}
	\centering
	\setlength{\unitlength}{1cm}
	\begin{tabular}{cc}	
	\begin{minipage}{.175\textwidth}
	\includegraphics{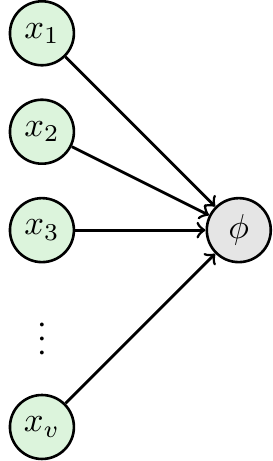}
	\end{minipage}
	&
	\begin{minipage}{.45\textwidth}
	\includegraphics[scale=1]{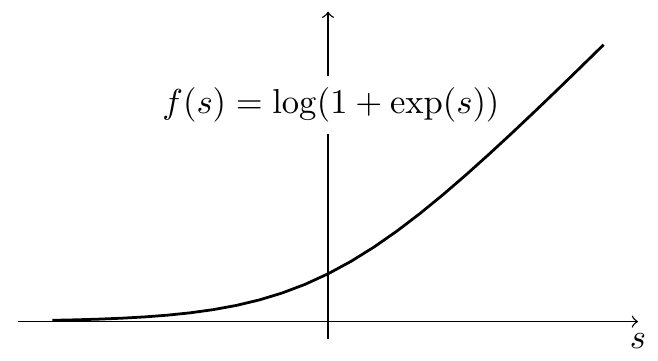}
	\end{minipage}
	\end{tabular}
	\caption{Illustration of a soft-plus computational unit. The possible inputs $\XX=\{0,1\}^V$, corresponding to the vertices of the unit $V$-cube, are mapped to the real line by an affine map $x\mapsto w^\top x + c$, and then the soft-plus non-linearity $f\colon s\mapsto\log(1+\exp(s))$ is applied. }
	\label{figure:softplus}
\end{figure}

What kinds of polynomials on $\{0,1\}^V$ can be represented by soft-plus units?  
Following Equation~\eqref{eq:moebius}, the polynomial coefficients of $\phi$ are given by 
\begin{equation}
K_{B} (w,c) 
= \sum_{C\subseteq B} (-1)^{|B \setminus C|} 
\log\left(1 + \exp \Big( \sum_{i\in C} w_{i}  + c \Big)\right),\quad B \in 2^V. 
\label{eq:coeffs}
\end{equation}
For each $B\in 2^V$ this is an alternating sum of the values $\phi(x)$ of the soft-plus unit on the input vectors $x\in\{0,1\}^V$ with $\supp(x)\subseteq B$. 
In particular, $K_B$ is independent of the parameters $w_i$, $i\not\in B$. 
We will use the shorthand notation $w_B$ for $(w_i)_{i\in B}$.  

Note that, if $w_i=0$ for some $i\in V$, then $K_C=0$ for all $C\in 2^V$ with $i\in C$. 
In the following we focus on the description of the possible values of the highest degree coefficients. 
For example, Younes~\cite{Younes1996109} showed that a soft-plus unit can represent a polynomial with an arbitrary leading coefficient: 

\begin{lemma}[Lemma~1 in~\cite{Younes1996109}]
\label{lem:Younes}
Let $B\subseteq V$ and $w_i=0$ for $i\not\in B$. 
%(by Proposition~\ref{proposition:1} this implies $K_{C}=0$ for all $C\not\subseteq B$). 
Then, for any $J_B\in \RR$, there is a choice of $w_B\in \RR^B$ and $c\in \RR$ such that $K_B = J_B$. 
\end{lemma}

The idea of Younes' proof of Lemma~\ref{lem:Younes} is to choose all non-zero $w_{i}$ of equal magnitude. This simplifies the calculations and reduces the number of free parameters to one. 
Our goal is to show that a soft-plus unit can actually freely model several polynomial coefficients at the same time. 
Our approach to simplify the M\"obius inversion formula~\eqref{eq:coeffs} is to choose the parameters $w$ and $c$ in such a way that the function $\phi$ has many zeros. 
Clearly this can only be done in an approximate way, since the soft-plus function is strictly positive. 
Nevertheless, these approximations can be made arbitrarily accurate, since $\log(1 + \exp(s)) \leq \exp(s)$ is arbitrarily close to zero for sufficiently large negative values of $s$. 

We call a pair of sets $(B,B')$ an \emph{edge pair} or a \emph{covering pair} when $B\supsetneq B'$ and there is no set $C$ with $B\supsetneq C\supsetneq B'$. 
The next lemma shows that a soft-plus unit can jointly model the coefficients of an edge pair, at least in part. 
When the maximum degree $|B|$ is at most~3, the two coefficients are restricted by an inequality, but when $|B|\geq 4$, there are no such restrictions. The result is illustrated in Figure~\ref{figure:region}. 

\begin{figure} 
\centering
\setlength{\unitlength}{1cm}
\begin{tabular}{ccc}
\begin{minipage}{.275\textwidth}
\includegraphics[scale=.75]{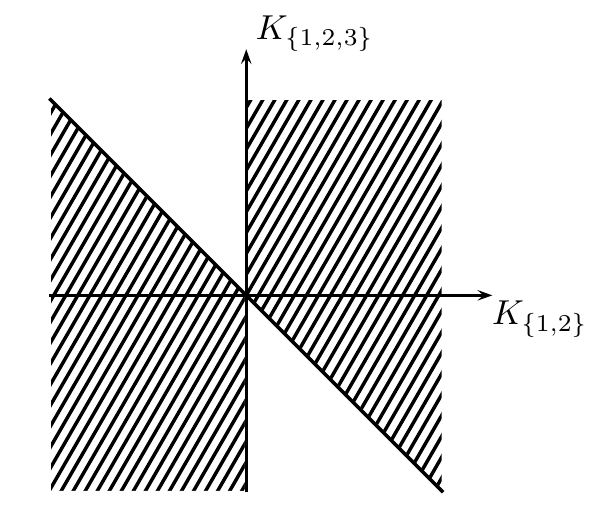}
\includegraphics[scale=.75]{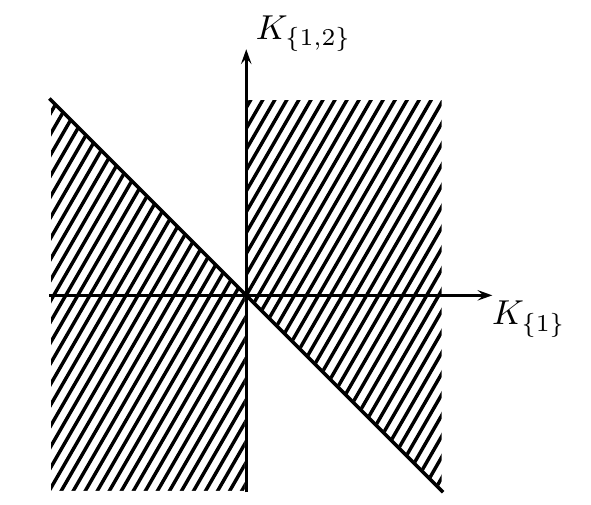}
\end{minipage}
&
\begin{minipage}{.35\textwidth}
\bigskip
\includegraphics[scale=.65]{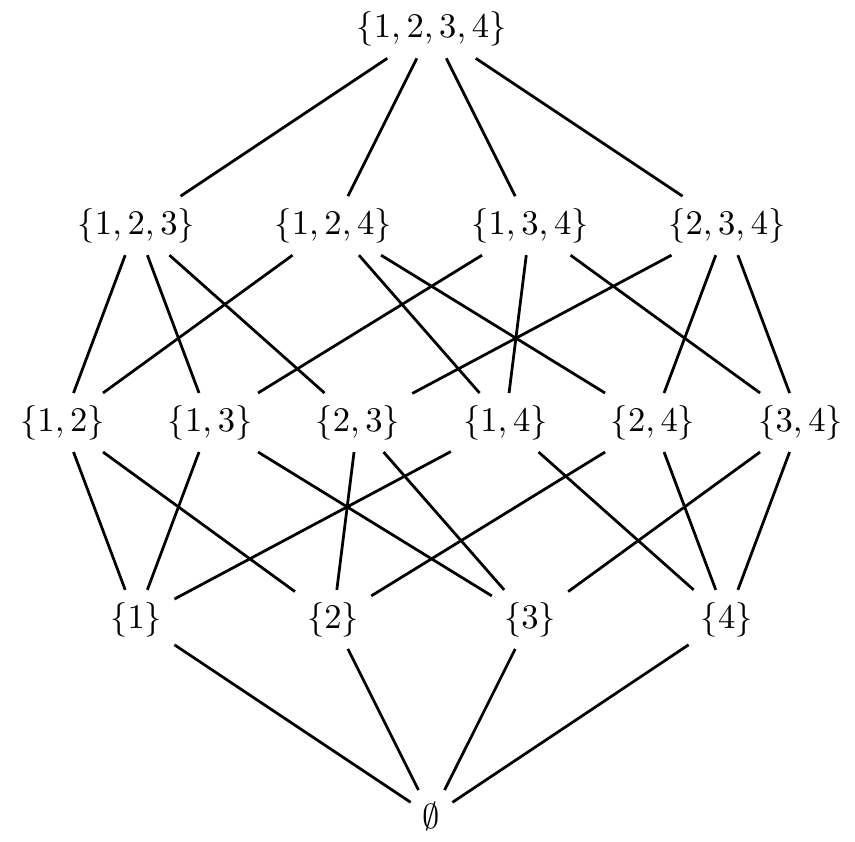}
\end{minipage}
&
\begin{minipage}{.275\textwidth}
\includegraphics[scale=.75]{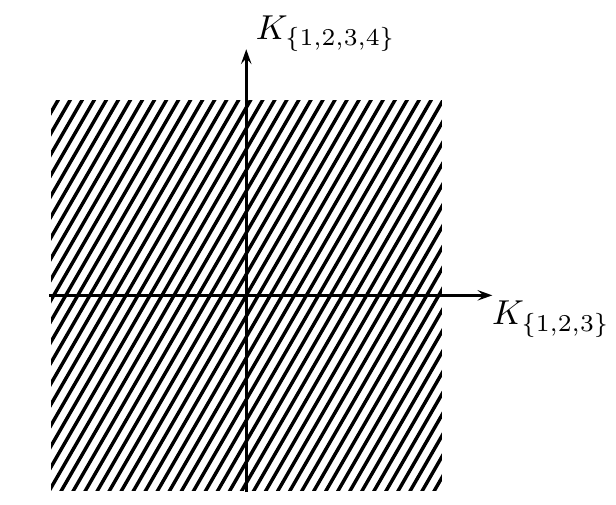}
\includegraphics[scale=.75]{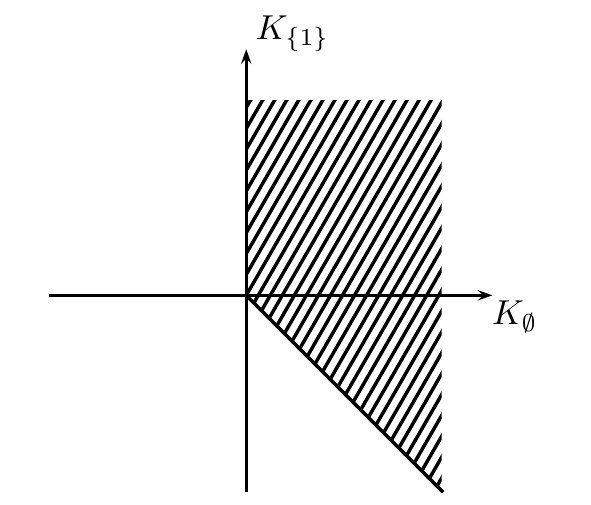}
\end{minipage}    
\end{tabular}
\caption{Illustration of Lemma~\ref{proposition:region}. 
Depicted is for each edge pair $(B,B')$ the set of coefficient pairs $(K_{B},K_{B'}) \in \RR^2$ of the polynomials $\sum_{C\subseteq V} K_C \prod_{i\in C} x_i$ expressible as $\log(1 + \exp(w^\top x + c))$. 
Shown is also the set of monomials of partial degree one and degree at most $4$, partially ordered by variable inclusion. 
}
\label{figure:region} 
\end{figure}

\begin{lemma}
\label{proposition:region}
	Consider an edge pair $(B, B')$. 
	%Let $w_i=0$ for $i\not\in B$ (this implies $K_{C}=0$ for all $C\not\subseteq B$). 
	Depending on $|B|$, for any $\epsilon>0$ there is a choice of $w_B\in\mathbb{R}^B$ and $c\in\mathbb{R}$ such that $\|(K_B, K_{B'}) - (J_B,J_{B'})\| \leq \epsilon$ 
	if and only~if
\begin{equation*}
\newcolumntype{R}{@{\extracolsep{.5cm}}r@{\extracolsep{.25cm}}}%
\renewcommand{\arraystretch}{1.2}
\begin{array}{cRl} 
		J_{B'}\geq 0,-J_B, 										 &\text{for}& |B|=1 \\
		J_{B'}\geq 0,-J_B \quad\text{or}\quad J_{B'}\leq 0,-J_B, &\text{for}& |B|=2 \\
		J_{B'}\geq 0,-J_B \quad\text{or}\quad J_{B'}\leq 0,-J_B, &\text{for}& |B|=3 \\ 
		(J_{B}, J_{B'})\in\RR^2, 								 &\text{for}& |B|\geq 4. 
\end{array}
\end{equation*}	 
\end{lemma}

\begin{proof}
This proof is deferred to Appendix~\ref{sec:proofs}. 
\end{proof}

\begin{remark}
If $(B,B')$ is an edge pair with $|B|=3$, then, despite having $|B|+1=4$ parameters to vary ($w_i$, $i\in B$, and $c$), we can only determine the polynomial coefficients $K_B$ and $K_{B'}$ up to a certain inequality. 
We expect that the same is true in general: If we want to freely control $k$ polynomial coefficients, we need strictly more than $k$ parameters. Otherwise, the coefficients are restricted by some inequalities. 
This situation is common in models with hidden variables. In particular, mixture models often require many more parameters to eliminate such inequalities than expected from na\"ive parameter counting~\cite{montufar2013mixture}. 
\end{remark}

It is natural to ask whether it is possible to control other pairs of coefficients or even larger groups of coefficients. 
We discuss a simple example before proceeding with the analysis of this problem. 

\begin{example}
\label{exampleV2}
Consider a soft-plus unit with two binary inputs. 
Write $f\colon s\mapsto\log(1+\exp(s))$ for the soft-plus non-linearity and $f_0 = f(c)$, $f_1=f(w_1+c)$, $f_2=f(w_2+c)$, $f_{12}=f(w_1+w_2+c)$ for the values of the soft-plus unit on $\{0,1\}^2$. 
From Equation~\eqref{eq:coeffs} it is easy to see that 
\begin{equation*}
\newcolumntype{Z}{@{\extracolsep{.1cm}}r@{\extracolsep{.1cm}}}%
\renewcommand{\arraystretch}{1.2}
\begin{array}{lZlZl}
K_\emptyset &=& f_0  	  &\geq& 0\\
K_{\{1\}} 	&=& f_1 - f_0 &\geq& -K_\emptyset\\
K_{\{2\}} 	&=& f_2 - f_0 &\geq& -K_\emptyset . 
\end{array}
\end{equation*}
Now let us investigate the quadratic coefficient $K_{\{1,2\}}=f_{12}-f_1-f_2+f_0$. Using the convexity of $f$ we find   
\begin{equation*}
\newcolumntype{R}{@{\extracolsep{.5cm}}r@{\extracolsep{.25cm}}}%
\newcolumntype{Z}{@{\extracolsep{.1cm}}r@{\extracolsep{.1cm}}}%
\renewcommand{\arraystretch}{1.2}
\begin{array}{rZcZl R rZl}
		 0&\leq&K_{\{1,2\}},& & 						  & \text{if}&    K_{\{1\}},K_{\{2\}}&\geq& 0 \\
		 0&\leq&K_{\{1,2\}} &\leq& -K_{\{1\}}, -K_{\{2\}}, & \text{if}& -K_{\{1\}}, -K_{\{2\}}&\geq& 0 \\
-K_{\{1\}}&\leq&K_{\{1,2\}} &\leq& 0, 			   		  & \text{if}&  K_{\{1\}}, -K_{\{2\}}&\geq& 0 \\
-K_{\{2\}}&\leq&K_{\{1,2\}} &\leq& 0, 			   		  & \text{if}& -K_{\{1\}},  K_{\{2\}}&\geq& 0 . 
\end{array}
\end{equation*}
Hence the computable polynomials have coefficient triples $(K_{\{1\}}, K_{\{2\}}, K_{\{1,2\}})$ enclosed in a polyhedral region of $\mathbb{R}^3$ as depicted in Figure~\ref{fig:V2example}. 
However, any pair $(K_{\{1\}},K_{\{2\}})\in\mathbb{R}^2$ is possible (for $K_{\emptyset}$ large enough). 

\begin{figure}
\centering
\includegraphics[clip=true,trim=0cm .5cm 0cm 1cm,scale=1]{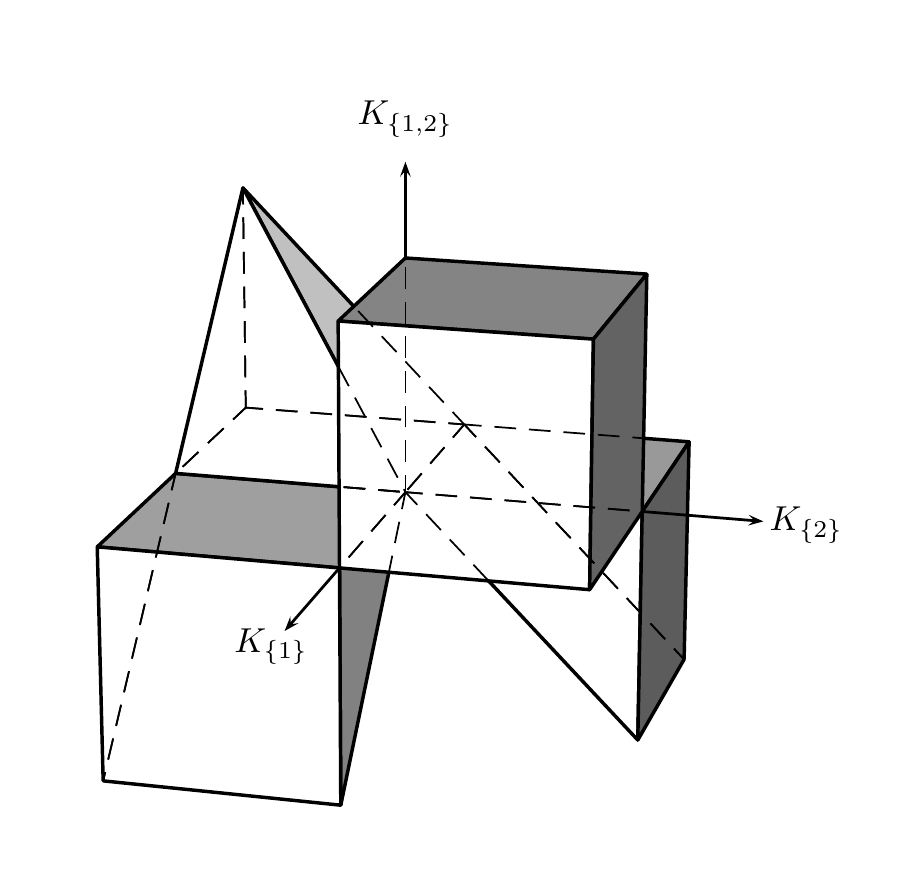}
\caption{Illustration of Example~\ref{exampleV2}. 
Depicted is a region of $\mathbb{R}^3$, clipped to $[-1,1]^3$, which contains the coefficient triples $(K_{\{1\}},K_{\{2\}},K_{\{1,2\}})\in\mathbb{R}^3$ of the polynomials computable by a soft-plus unit with two binary inputs. 
This region consists of $4$ solid convex cones. }
\label{fig:V2example}
\end{figure}
\end{example}

The next lemma %Lemma~\ref{generallemma} 
shows that a soft-plus unit can jointly model certain tuples of polynomial coefficients corresponding to $\nv-k+1$ monomials of degree $k$. 
We call \emph{star tuple} a set of the form $\{B\cup\{j\} \colon j\in B'\}$, where $B,B'\subseteq V$ satisfy $B\cap B'=\emptyset$. 
Each element of the star tuple covers the set $B$. 
In the Hasse diagram of the power set $2^V$, the sets $B\cup\{j\}, j\in B'$, are the leaves of a star with root $B$. 

\begin{lemma}
\label{generallemma}
Consider any $B,B'\subseteq V$ with $B\cap B'=\emptyset$. 
Let $w_i=0$ for $i\not\in B\cup B'$. 
Then, for any $J_{B\cup\{j\}}\in\mathbb{R}$, $j\in B'$, and $\epsilon>0$, there is a choice of $w_{B\cup B'}\in\mathbb{R}^{B\cup B'}$ and $c\in\mathbb{R}$ such that $|K_{B\cup\{j\}} - J_{B\cup\{j\}} |\leq \epsilon$ for all $j\in B'$, and $|K_C|\leq \epsilon$ for all $C\neq B, B \cup\{j\}$, $j\in B'$. 
\end{lemma}

\begin{proof}[Proof of Lemma~\ref{generallemma}]
Since $w_i=0$ for $i\not\in B\cup B'$, we have that $K_{C}=0$ for all $C\not\subseteq B\cup B'$. 
We choose $c=-(|B|-\frac12)\omega$, $w_i = \omega$ for all $i\in B$, and $w_j = J_{B\cup\{j\}}$ for $j\in B'$. 
Choosing $\omega\gg \sum_{j\in B'}|w_j|$ yields $f(\sum_{i\in C}w_i +c) \approx 0$ for all $C\not\supseteq B$. 
In this case, 
\begin{equation*}
K_C \approx 0, \quad \text{for all }B\not\subseteq C\subseteq B\cup B'. 
\end{equation*}
Furthermore, for all $j\in B'$ we have  
\begin{align*}
K_{B\cup\{j\}}
\approx& f\Big(\sum_{i\in B} w_i + w_j + c\Big) - f\Big(\sum_{i\in B} w_i + c\Big)\\ 
=& f(J_{B\cup\{j\}} + \tfrac12\omega) - f(\tfrac12\omega ) \\
\approx& (J_{B\cup\{j\}} + \tfrac12\omega) - (\tfrac12\omega) 
= J_{B\cup\{j\}}.  
\end{align*}
Similarly, $K_{B\cup C}\approx 0$ 
for all $C\subseteq B'$ with $|C|\geq 2$. 
Note that $K_B\approx \frac12\omega$. 
\end{proof}

The intuition behind Lemma~\ref{generallemma} is simple. 
When $\sum_{i\in B} w_i + c\gg 1$, 
the values $w^\top x +c$, for $x$ with $x_i=1$, $i\in B$, 
fall in a region where the soft-plus function is nearly linear. 
In turn, the soft-plus unit is nearly a linear function of $x_j$, $j\in B'$, with coefficients $w_j$, $j\in B'$.

\begin{remark}
Closely related to soft-plus units are \emph{rectified linear units}, which compute functions of the form 
\begin{equation*}
\varphi\colon \{0,1\}^V\to \mathbb{R};\; x\mapsto \max\{ 0, w^\top x +c \}. 
\end{equation*}
In this case the non-linearity is $s\mapsto\{0, s\}$. 
This reflects precisely the zero/linear behavior of the soft-plus activation for large negative or positive values of $s$. 
Our polynomial descriptions are based on this behavior and hence they apply both to soft-plus and rectified linear units. 
\end{remark}

We close this section with a brief discussion of dependencies among coefficients. 
The next proposition gives a perspective on the possible values of the coefficient $K_{B}$, depending on $w_m$, once $w_{B\setminus\{m\}}$ and $c$ have been fixed. 

\begin{proposition}
\label{proposition:1a}
Let $(B,B')$ be an edge pair with $B'=B\setminus\{m\}$ and let $J_B\in\mathbb{R}$. 
For fixed $w_{B'}\in \RR^{B'}$ and $c\in \RR$, 
there is some $w_m\in\mathbb{R}$ such that $K_{B}=J_B$ 
if and only if a certain 
degree-$2^{|B'|-1}$ polynomial in one real variable has a positive root. 
\end{proposition}
%\todo{finish this}
\begin{proof}[Proof of Proposition~\ref{proposition:1a}]
Observe that
\begin{equation*}
	K_{B}(w,c) = K_{B'}(w_{B'}, c+ w_{m}) - K_{B'}(w_{B'}, c). 
\end{equation*}
Hence $K_B = J_B$ if and only if $K_{B'}(w_{B'}, c+ w_{m}) = K_{B'}(w_{B'}, c)  + J_B =:r$. 
We use the abbreviation $\tilde t = e^t$, which implies positivity. 
We have 
\begin{align*}
K_{B'}(w_{B'},c + w_m) 
&= \sum_{C\subseteq B'} (-1)^{|B'\setminus C|} \log\left(1 + \exp\Big(\sum_{i\in C} w_i + c + w_m\Big)\right) \\
&= \log\left(\prod_{C\subseteq B'}  \Big(1 + \tilde w_m \tilde c\prod_{i\in C} \tilde w_i \Big)^{(-1)^{|B'\setminus C|}} \right). 
\end{align*}
Now, $K_{B'}(w_{B'},c + w_m) = r$ if and only if 
\begin{gather*}
\prod_{C\subseteq B'}  \Big(1 + \tilde{w_m} \tilde c\prod_{i\in C} \tilde w_i \Big)^{(-1)^{|B'\setminus C|}} = \tilde r,  
\intertext{or, equivalently,}  
\prod_{ \substack{C\subseteq B'\colon \\ B'\setminus C\text{ even} }}\!\! \Big(1 + \tilde w_m \tilde c\prod_{i\in C} \tilde w_i \Big)  \;\;-\;\; \tilde r \!\! \prod_{\substack{C\subseteq B'\colon \\ B'\setminus C\text{ odd} }} \!\!\Big(1 + \tilde w_m \tilde c\prod_{i\in C} \tilde w_i \Big) = 0. 
\end{gather*} 
This is a polynomial of degree at most $2^{|B'|-1}$ in $\tilde w_m=e^{w_m}$. 
\end{proof}

This description implies various kinds of constraints. 
For example, by Descartes' rule of signs, a polynomial can only have positive roots if the sequence of polynomial coefficients, ordered by degree, has sign changes. 

\section{Conditionally Independent Hidden Variables}
\label{section:RBMs}

In the case of a bipartite graph between $V$ and $H$ with all variables binary, 
the hierarchical model (or rather its visible marginal) is a restricted Boltzmann machine, denoted $\RBM_{V,H}$. This model is illustrated in Figure~\ref{fig:spnetw}. 
The free energy takes the form 
\begin{equation*}
F(x) 
=  
\sum_{i \in V} b_i x_i 
+
\sum_{j\in H}
\log\left(1 + \exp \Big(\sum_{i\in V} w_{ji} x_i  + c_j\Big)\right). 
\end{equation*}
This is the sum of an arbitrary degree-one polynomial, with coefficients $b_i\in\mathbb{R}$, $i\in V$, and $\nh=|H|$ independent soft-plus units, 
with parameters $w_{ji}\in\mathbb{R}$, $j\in H, i\in V$ and $c_j\in\mathbb{R}$, $j\in H$. 
The free energy contributed by the hidden variables can be thought of as a feedforward network with soft-plus computational units. 

\begin{figure}
\centering
\includegraphics{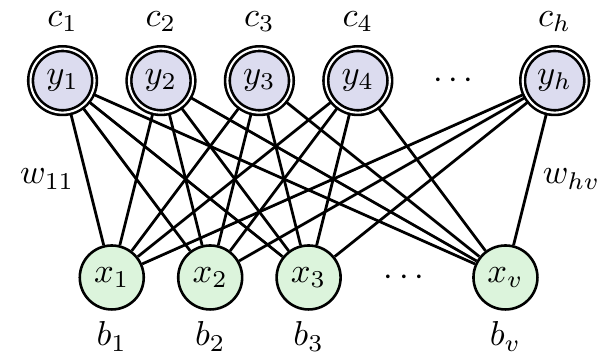}
\caption{A restricted Boltzmann machine. The free energy contributed by the hidden units is a sum of independent soft-plus units. 
}
\label{fig:spnetw}
\end{figure}

We can use each soft-plus unit to model a group of coefficients of any given polynomial, starting at the highest degrees. 
Using the results from Section~\ref{section:softplus} we arrive at the following representation result:  

\begin{theorem}\label{theorem1a}
Every distribution from a hierarchical model $\Ecal_S$ on $\{0,1\}^V$ can be approximated arbitrarily well by distributions from $\RBM_{V,H}$ whenever there exist $\nh$ sets $\bfB_1,\ldots, \bfB_{\nh}\subseteq 2^V$ which cover $\{\Lambda\in S\colon |\Lambda|\geq 2\}$ in reverse inclusion order, 
where each $\bfB_j$ is a star tuple or an edge pair of sets of cardinality at least $3$. 
\end{theorem}

\begin{proof}[Proof of Theorem~\ref{theorem1a}]
We need to express the possible energy functions of the hierarchical model as sums of independent soft-plus units plus linear terms. 
This problem can be reduced to covering the appearing monomials of degree two or more by groups of coefficients that can be jointly controlled by soft-plus units. 
In view of Lemmas~\ref{proposition:region} and~\ref{generallemma}, 
edge pairs with sets of cardinality $3$ or more and star tuples can be jointly controlled. 
We start with the highest degrees and cover monomials downwards, 
because setting the coefficients of a given group may produce uncontrolled values for the coefficients of smaller monomials. 
Since $S$ is a simplicial complex, we only need to cover the elements of $S$. 
\end{proof}

Finding a minimal covering is in general a hard combinatorial problem. 
In the following we derive upper bounds for the $k$-interaction model, which is the hierarchical model $\Ecal_S$ with $S=S_k:=\{\Lambda\subseteq V\colon |\Lambda|\leq k\}$. 
We will focus on star tuples and consider individual coverings of the layers $\binom{V}{j} = \{\Lambda\subseteq V\colon |\Lambda| = j\}$. 
Let $\nv=|V|$. Denote $D(\nv,j)$ the smallest number of star tuples that cover $\binom{V}{j}$. 
We use the following notion from the theory of combinatorial designs (see~\cite{Colbourn:2006:HCD:1202540} for an overview on that subject). For integers $\nv\geq k>r$ denote $C(\nv,k,r)$ the smallest possible number of elements of $\binom{V}{k}$ such that every element from $\binom{V}{r}$ is contained in at least one of them. 

\begin{lemma}
\label{lemma:covering}
For $0< j \leq \nv$, 
the minimal number of star tuples that cover $\binom{V}{j}$ 
is $D(\nv,j) = C(\nv,\nv-j+1,\nv-j)$. 
Inserting known results for $C(\nv,t+1,t)$ we obtain 
the exact values 
\begin{equation*}
\newcolumntype{Z}{@{\extracolsep{.2cm}}r@{\extracolsep{.2cm}}}%
\renewcommand{\arraystretch}{1.2}
\begin{array}{lZl}
D(\nv,1) %= C(\nv,\nv,\nv-1) 
&=& 1 \\
D(\nv,2) %= C(\nv,\nv-1,\nv-2) 
&=& \nv-1 \\ 
D(\nv,3) %= C(\nv,\nv-2,\nv-3) 
&=& 
%\left\lceil \frac{\nv}{\nv-2}\left\lceil\frac{\nv-1}{\nv-3} \cdots\left\lceil\frac{\nv-(\nv-3)+1}{(\nv-2)-(\nv-3)+1}\right\rceil\cdots\right\rceil \right\rceil \\
\left\lceil \frac{\nv}{\nv-2}\left\lceil\frac{\nv-1}{\nv-3} \cdots\left\lceil\frac{4}{2}\right\rceil\cdots\right\rceil \right\rceil \\
D(\nv,\nv-3) %= C(\nv,4,3) 
&=& \left\lceil \frac{\nv}{4}\left\lceil\frac{\nv-1}{3}\left\lceil\frac{\nv-2}{2}\right\rceil\right\rceil \right\rceil \quad (\nv\not\equiv 7\operatorname{mod}12)\\
D(\nv,\nv-2) %= C(\nv,3,2) 
&=& \left\lceil\frac{\nv}{3}\left\lceil\frac{\nv-1}{2}\right\rceil\right\rceil \\
D(\nv,\nv-1) %= C(\nv,2,1) 
&=&\left\lceil\frac{\nv}{2}\right\rceil \\
D(\nv,\nv) %= C(\nv,1,0) 
&=& 1  
\end{array}
\end{equation*}
and the general bound 
\begin{equation*}
D(\nv,j) %= C(\nv,\nv-j+1,\nv-j)
\leq \frac{1+\log(\nv-j+1)}{\nv-j+1}\binom{\nv}{j}, \quad 0<j\leq \nv. 
\end{equation*}
Furthermore, we have the simple bound $D(\nv,j)\leq \binom{\nv-1}{j-1}$, $0<j\leq \nv$. 
\end{lemma}

\begin{proof}[Proof of Lemma~\ref{lemma:covering}]
A star tuple covering of $\binom{V}{j}$ is given by a collection $B_1,\ldots, B_n$ of elements of $\binom{V}{j-1}$ such that every element of $\binom{V}{j}$ contains at least one of the $B_i$. 
The minimal possible number of elements in such a collection is precisely $D(\nv,j) = C(\nv,\nv-j+1,\nv-j)$. 
The equalities follow from corresponding equalities for $C(\nv,\nv-j+1,\nv-j)$ by several authors, which are listed in~\cite{JCD:JCD5}. 
The inequality follows from a result by Erd\H{o}s and Spencer~\cite{erdosspencer} showing that 
$C(v,k,r)\leq \left[ \binom{v}{r}\middle/\binom{k}{r} \right] \left[1 + \log\binom{k}{r}\right]$. 
The simple bound results from the fact that each set $B$ from $\binom{V}{j}$ contains a set $B'$ from $\binom{V\setminus\{1\}}{j-1}$. 
\end{proof}

\begin{remark}
Lemma~\ref{lemma:covering} presents widely applicable bounds on the cardinality of star tuple coverings, which are naturally not always tight. For $\nv\leq 28$, better individual bounds on $C(\nv,t+1,t)$ can be found in~\cite[Table~III]{JCD:JCD5}. 
See also~\cite{oeis} for a list of known exact values. 
In another direction~\cite{Roedl198569} offers optimal asymptotic bounds on $C(v,k,r)$ for fixed $k$ and $r$. 
\end{remark}

Lemma~\ref{lemma:covering} allows us to formulate the following more explicit version of Theorem~\ref{theorem1a}: 

\begin{theorem}
\label{theorem2a}
Let $1\leq k\leq \nv$. 
Every distribution from the $k$-interaction model $\Ecal_{S_k}$ on $\{0,1\}^V$ can be approximated arbitrarily well by distributions from $\RBM_{V,H}$ whenever $\nh$ surpasses or equals $U(\nv,k)=\sum_{j=2}^k D(\nv,j)$, which is bounded above as indicated in Lemma~\ref{lemma:covering}. 
This is the case, in particular, whenever $\nh\geq\sum_{j=2}^k \binom{\nv-1}{j-1}$ or 
$\nh\geq \frac{\log(\nv-1)+1}{\nv+1}\sum_{j=2}^k\binom{\nv+1}{j}$. 
\end{theorem}

\begin{proof}[Proof of Theorem~\ref{theorem2a}]
This follows directly from Theorem~\ref{theorem1a} and Lemma~\ref{lemma:covering}. 
For the last statement we use the simple bound from the lemma, by which $D(\nv,j)\leq \binom{\nv-1}{j-1}$, 
and the general bound, by which $D(\nv,j) \leq \frac{\log(\nv-j+1)+1}{\nv+1}\binom{\nv+1}{j}$.  
\end{proof}

In order to provide a numerical sense of Theorem~\ref{theorem2a} we give upper bounds on $U(\nv,k)$, $2\leq k\leq \nv\leq 14$, in Table~\ref{table:kinteraction}. 
For convenience we also provide an Octave~\cite{octave} script for computing such bounds in~\href{http://personal-homepages.mis.mpg.de/montufar/}{http://personal-homepages.mis.mpg.de/montufar/starcover.m}.

\begin{table}
\centering
\newcommand{\ob}[2]{$\underset{\tiny #1}{#2}$} % old univ approx bound in tiny font below
\newcommand{\ib}[2]{\ensuremath{\overset{\tiny #2}{#1}}} % from general bounds in tiny font above
\newcommand{\tbs}{\textbackslash}
\begin{tabular}{c|ccccccccccccc}
$k$\tbs $v$ & 2 	  & 3 		& 4 	  & 5 		  & 6 		& 7 		& 8 				  & 9  		  			  & 10 					  & 11 					   & 12 					 & 13  						& 14\\\hline
2 			&\ob{1}{1}& 2 		& 3 	  & 4 		  & 5 		& 6 		& 7 				  & 8  		  			  & 9  					  & 10 					   & 11 				     & 12  						& 13\\
3 			& - 	  &\ob{3}{3}& 5 	  & 8 		  & 11		& 15		& 19				  & 24 		  		  	  & 29 					  & 35 				   	   & 41  					 & 48  						& 55\\
4			& - 	  & - 		&\ob{7}{6}& 11		  & 17		& 27		& \ib{39}{55}		  & \ib{54}{82} 		  &\ib{74}{117}			  & \ib{98}{162}		   & \ib{125}{216} 			 & 268 						& 341\\
5			& - 	  & - 		& - 	  &\ob{15}{12}& 20		& 34		& \ib{53}{69}		  & \ib{84}{147}		  &\ib{124}{234}		  & \ib{182}{356}		   & \ib{251}{520} 			 & \ib{453}{725} 			& 1002\\
6			& - 	  & - 		& - 	  & - 		  &\ob{31}{21}& 38		& \ib{64}{80}		  &\ib{109}{172}		  &\ib{175}{343}		  & \ib{282}{570}		   & \ib{427}{908} 			 & \ib{750}{1385}			& \ib{1473}{2068}\\
7			& - 	  & - 		& - 	  & - 		  & - 		&\ob{63}{39}& \ib{68}{84}		  &\ib{121}{184}  		  &\ib{205}{373} 		  & \ib{348}{742}		   & \ib{559}{1276}			 & \ib{1014}{2107}			& \ib{1944}{3389}\\
8			& - 	  & - 		& - 	  & - 		  & - 		& - 		&\ob{127}{\ib{69}{85}}&\ib{126}{189}		  &\ib{222}{390}		  & \ib{395}{789}		   & \ib{672}{1534}			 & \ib{1259}{2705}			& \ib{2452}{4652}\\
9			& - 	  & - 		& - 	  & - 		  & - 		& - 		& - 				  &\ob{255}{\ib{127}{190}}&\ib{227}{395}		  & \ib{414}{808}		   & \ib{729}{1591}			 & \ib{1416}{3078}			& \ib{2823}{5583}\\
10			& - 	  & - 		& - 	  & - 		  & - 		& - 		& - 				  & -  					  &\ob{511}{\ib{228}{396}}& \ib{420}{814}		   & \ib{753}{1615}			 & \ib{1494}{3156}			& \ib{3053}{6105}\\
11			& - 	  & - 		& - 	  & - 		  & - 		& - 		& - 				  & -  					  & -  					  &\ob{1023}{\ib{421}{815}}& \ib{759}{1621}			 & \ib{1520}{3182}			& \ib{3144}{6196}\\
12			& - 	  & - 		& - 	  & - 		  & - 		& - 		& - 				  & -  					  & -  					  & -  					   &\ob{2047}{\ib{760}{1622}}& \ib{1527}{3189}			& \ib{3177}{6229}\\
13			& - 	  & - 		& - 	  & - 		  & - 		& - 		& - 				  & -  					  & -  					  & -  					   & -   					 &\ob{4095}{\ib{1528}{3190}}& \ib{3184}{6236}\\
14			& - 	  & - 		& - 	  & - 		  & - 		& - 		& - 				  & -  					  & -  					  & -  					   & -   					 & -   						& \ob{8191}{\ib{3185}{6237}}
\end{tabular}
\caption{
Upper bounds on the minimal number of hidden units for which $\RBM_{V,H}$ can approximate every distribution from the $k$-interaction model $\Ecal_{S_k}$ on $\{0,1\}^V$ arbitrarily well, following from Theorem~\ref{theorem2a}, for $2\leq k\leq \nv\leq 14$. 
Shown are upper bounds on $U(\nv,k)=\sum_{j=2}^k D(\nv,j)$ evaluated using Lemma~\ref{lemma:covering} and some individual bounds on $D(\nv,j)=C(\nv,\nv-j+1,\nv-j)$ from~\cite[Table~III]{JCD:JCD5}. 
Upper scripts indicate values obtained using only Lemma~\ref{lemma:covering}. 
Lower scripts indicate the previous RBM universal approximation bound $2^{\nv-1}-1$ from~\cite{montufar2011refinements}. 
Entries with $v\leq 9$ or $k\leq 3$ are exact values of $U(\nv,k)$. 
}
\label{table:kinteraction}
\end{table}

In the special case $k=\nv$, the $k$-interaction model $\Ecal_{S_k}$ is the \emph{full interaction model} and contains all (strictly positive) probability distributions on $\{0,1\}^V$. 
Hence Theorem~\ref{theorem2a} entails the following universal approximation result: 

\begin{corollary}
\label{corollary:universal}
Every distribution on $\{0,1\}^V$ can be approximated arbitrarily well by distributions from $\RBM_{V,H}$ whenever $\nh$ surpasses or equals $U(\nv,\nv)$, which is bounded above as indicated in Lemma~\ref{lemma:covering}. This is the case, in particular, whenever $\nh\geq 2^{\nv-1}-1$ or 
$\nh\geq \frac{2(\log(\nv-1)+1)}{\nv+1}(2^{\nv} -(\nv+1) -1) +1$. 
\end{corollary}

Corollary~\ref{corollary:universal} provides a significant and unexpected improvement the best previously known upper bound $2^{\nv-1}-1$ from~\cite{montufar2011refinements}. 
Whether the upper bound $2^{\nv -1}-1$ was optimal or not had remained an open problem in~\cite{montufar2011refinements} and several succeeding papers. 
In Table~\ref{table:universal} we give upper bounds on $U(\nv,\nv)$, $2\leq \nv\leq 40$, and compare these with the previous result.

\begin{table} 
\centering
\scalebox{1}{ 
\begin{tabular}{r|r|r|r}
    $v$        & $U(v,v)\leq$      & $2^{v-1}-1=$   &$\left\lceil\frac{2^v}{v+1}-1\right\rceil=$ \\ \hline
              2&              1&              1&              1\\
              3&              3&              3&              1\\
              4&              6&              7&              3\\
              5&             12&             15&              5\\
              6&             21&             31&              9\\
              7&             39&             63&             15\\
              8&             69&            127&             28\\
              9&            127&            255&             51\\
             10&            228&            511&             93\\
             11&            421&           1023&            170\\
             12&            760&           2047&            315\\
             13&           1528&           4095&            585\\
             14&           3185&           8191&           1092\\
             15&           6642&          16,383&           2047\\
             16&          14,269&          32,767&           3855\\
             17&          30,352&          65,535&           7281\\
             18&          63,431&         131,071&          13,797\\
             19&         132,195&         262,143&          26,214\\
             20&         272,160&         524,287&          49,932\\
             21&         553,195&        1048,575&          95,325\\
             22&        1115,207&        2097,151&         182,361\\
             23&        2227,484&        4194,303&         349,525\\
             24&        4427,830&        8388,607&         671,088\\
             25&        8760,826&       16,777,215&        1290,555\\
             26&       17,265,199&       33,554,431&        2485,513\\
             27&       33,951,316&       67,108,863&        4793,490\\
             28&       66,656,315&      134,217,727&        9256,395\\
             29&      132,084,407&      268,435,455&       17,895,697\\
             30&      257,962,181&      536,870,911&       34,636,833\\
             31&      504,141,876&     1073,741,823&       67,108,863\\
             32&      985,875,453&     2147,483,647&      130,150,524\\
             33&     1929,093,753&     4294,967,295&      252,645,135\\
             34&     3776,867,237&     8589,934,591&      490,853,405\\
             35&     7398,516,744&    17,179,869,183&      954,437,176\\
             36&    14,500,416,431&    34,359,738,367&     1857,283,155\\
             37&    28,433,369,622&    68,719,476,735&     3616,814,565\\
             38&    55,779,952,400&   137,438,953,471&     7048,151,460\\
             39&   109,476,401,847&   274,877,906,943&    13,743,895,347\\
             40&   214,954,581,277&   549,755,813,887&    26,817,356,775
\end{tabular}
}
\caption{
Bounds on the minimal number of hidden units for which $\RBM_{V,H}$ can approximate every distribution on $\{0,1\}^V$ arbitrarily well, for $2\leq \nv\leq 40$. 
The first column gives upper bounds following from Corollary~\ref{corollary:universal}. 
Shown are upper bounds on $U(\nv,\nv)=\sum_{j=2}^v D(\nv,j)$ evaluated using Lemma~\ref{lemma:covering} and some individual bounds on $D(\nv,j)=C(\nv,\nv-j+1,\nv-j)$ from~\cite[Table~III]{JCD:JCD5}. 
The second column gives the previous upper bound $2^{\nv-1}-1$ from~\cite{montufar2011refinements}. 
The last column gives the hard lower bound $\left\lceil \frac{2^\nv}{\nv+1}-1\right\rceil$ that results from parameter counting, i.e., from demanding that the model $\RBM_{V,H}$ has at least $(\nh+1)(\nv+1)-1 \geq 2^\nv-1$ parameters. 
}
\label{table:universal}
\end{table}

\begin{remark}
In general an RBM can represent many more distributions than just the interaction models described above. 
For several small examples discussed further below, our bounds for the representation of interaction models are tight. 
However, Theorem~\ref{theorem2a} is based on upper bounds on a specific type of coverings and we suspect that it can be further improved, at least in some special cases, even if not reaching the hard lower bounds coming from parameter counting. 
\end{remark}

Besides from RBMs we can also consider models that include interactions among the visible variables other than biases. In this case we only need to cover the interaction sets from the simplicial complex $S$ that are not already included in the simplicial complex $T$. 
In Theorem~\ref{theorem1a} one just replaces $\{\Lambda\in S\colon |\Lambda|\geq 2\}$ by $S\setminus T$. 
We note the following special case: 

\begin{corollary}\label{corollary2aa}
Every distribution from the $k$-interaction model $\Ecal_{S_k}$ on $\{0,1\}^V$ can be approximated arbitrarily well by the visible marginals of a pairwise interaction model with 
$\nh = \sum_{j=3}^k D(\nv,j)$ hidden binary variables. 
The latter is bounded above as indicated in Lemma~\ref{lemma:covering}. 
In particular, every distribution on $\{0,1\}^V$ can be approximated arbitrarily well by the visible marginals of a pairwise interaction model with $\nh=2^{\nv-1}-(\nv-1)-1$ or $\nh=\frac{2(\log(\nv-2)+1)}{\nv+1}(2^{\nv}-(\nv+1)-1 -\frac{(\nv+1)\nv}{4})+1$ hidden binary variables. 
\end{corollary}

Corollary~\ref{corollary2aa} improves a previous result by Younes~\cite{Younes1996109}, which showed that a pairwise interaction model with $\nh=2^\nv -\binom{\nv}{2}-\nv-1$ hidden binary variables can approximate every distribution on $\{0,1\}^V$ arbitrarily well.

We close this section with a few small examples illustrating our results. 

\begin{example} %[$\RBM_{3,1}$] 
The model $\RBM_{3,1}$ is the same as the two-mixture of product distributions of $3$ binary variables and is also known as the \emph{tripod tree model}. 
It has $7$ parameters and the same dimension. 
What is the largest hierarchical model contained in the closure of this model? 
The closure of a model $\mathcal{M}$ is the set of all probability distributions that can be approximated arbitrarily well by probability distributions from $\mathcal{M}$. 

The closure of $\RBM_{3,1}$ contains all $3$ hierarchical models on $\{0,1\}^3$ with two pairwise interactions. For example, it contains the model $\Ecal_{S}$ with $S=\{\{1,2\},\{1,3\},\{1\},\{2\},\{3\}\}$. 
Indeed, two quadratic coefficients can be jointly modeled by one soft-plus unit (Lemma~\ref{generallemma}) and the linear coefficients with the biases of the visible variables. 
In particular, the closure of $\RBM_{3,1}$ also contains the $3$ hierarchical models with a single pairwise interaction.   

It does not contain the hierarchical model with $3$ pairwise interactions, $\Ecal_{S}$ with $S=S_2=\{\{1,2\},\{1,3\},\{2,3\},\{1\},\{2\},\{3\}\}$, which is known as the \emph{no-three-way interaction model}. 
One way of proving this is by comparing the possible support sets of the two models, 
as proposed in~\cite{montufar2013mixture}. 
The support set of a product distribution is a cylinder set. 
The support set of a mixture of two product distributions is a union of two cylinder sets. 
On the other hand, the possible support sets of a hierarchical model correspond to the faces of its marginal polytope, $\operatorname{conv}\{(\prod_{i\in\Lambda}x_i)_{\Lambda\in S}\colon x\in \XX \}\subset \mathbb{R}^S$. 
The marginal polytope of the no-three-way interaction model is the cyclic polytope $C(N,d)$ with $N=8$ vertices and dimension $d=6$ 
(see, e.g.,~\cite[Lemma~18]{montufar2013mixture}). 
This is a neighborly polytope, meaning that every $d/2 = 3$ or less vertices form a face. 
In turn, every subset of $\{0,1\}^3$ of cardinality $d/2=3$ is the support set of a distribution in the closure of the no-three-way interaction model.\footnotemark 
Since the set $\{(100), (010), (001)\}$ is not a union of two cylinder sets, the closure of $\RBM_{3,1}$ does not contain the no-three-way interaction model. 
\end{example}
\footnotetext{More generally, in~\cite{Kahle2010} it is shown that if $S\supseteq\{\Lambda\subseteq V\colon |\Lambda|\leq k\}$, then the marginal polytope of $\Ecal_{S}$ is $2^k -1$ neighborly, meaning that any $2^k -1$ or fewer of its vertices define a face.}

\begin{example} %[$\RBM_{3,2}$]
The closure of $\RBM_{3,2}$ contains the no-three-way interaction model. 
Two of the quadratic coefficients can be jointly modeled with one hidden unit and the third with the second hidden unit (Lemma~\ref{generallemma}). 
%and the linear coefficients with the biases of the visible units. 

It does not contain the full interaction model. 
Following the ideas explained in the previous example, this can be shown by analyzing the possible support sets of the distributions in the closure of $\RBM_{3,2}$. 
For details on this we refer the reader to~\cite{montufar2012does}. 
\end{example}

\begin{example} %[$\RBM_{3,3}$] 
The model $\RBM_{3,3}$ is a universal approximator. 
This follows immediately from the universal approximation bound $2^{v-1}-1$ from~\cite{montufar2011refinements}. 
This observation can be recovered from our results as follows. 
The cubic coefficient can be modeled with one hidden unit (Lemma~\ref{lem:Younes}). 
Two quadratic coefficients can be jointly modeled with one hidden unit and the third with another hidden unit (Lemma~\ref{generallemma}). 
%Finally the linear coefficients can be modeled with the biases of the visible variables. 
\end{example}

\begin{example} %[$\RBM_{4,6}$]
%Previous results show that $\RBM_{4,7}$ is a universal approximator; see~\cite{montufar2011refinements}. 
The model $\RBM_{4,6}$ is a universal approximator. 
The quartic coefficient can be modeled with one hidden unit. 
The $4$ cubic coefficients can be modeled with two hidden units (Lemma~\ref{generallemma}). 
The $6$ quadratic coefficients can be grouped into $3$ pairs with a shared variable in each pair. 
These can be modeled with $3$ hidden units (Lemma~\ref{generallemma}). 
\end{example}

\section{Conclusions and Outlook}
\label{section:conclusions}

We studied the kinds of interactions that appear when marginalizing over a hidden variable that is connected by pair-interactions with all visible variables. 
We derived upper bounds on the minimal number of variables of a hierarchical model whose visible marginal distributions can approximate any distribution from a given fully observable hierarchical model arbitrarily well. 
These results generalize and improve previous results on the representational power of RBMs from~\cite{montufar2011refinements} and~\cite{Younes1996109}. 

Many interesting questions remain open at this point: 
A full characterization of soft-plus polynomials and the necessary number of hidden variables is missing. 

It would be interesting to look at non-binary hidden variables. 
This corresponds to analyzing the hierarchical models that can be represented by mixture models. 
In the case of conditionally independent binary hidden variables, the partial factorization leads to soft-plus units, whereas in the case of higher-valued hidden variables, it leads to shifted logarithms of denormalized mixtures. 
Similarly, it would be interesting to take a look at non-binary visible variables. In this case state vectors cannot be identified in a one-to-one manner with subsets of variables. This means that the correspondence between function values and polynomial coefficients is not as direct. 

Our analysis could also be extended to cover the representation of conditional probability distributions from hierarchical models in terms of conditional restricted Boltzmann machines and to refine the results on this problem reported in~\cite{montufar2014expressive}. 

Another interesting direction are models where the hidden variables are not conditionally independent given the visible variables, such as deep Boltzmann machines, which involve several layers of hidden variables. 
This case is more challenging, since the free energy does not decompose into independent terms.

\subsection*{Acknowledgments}
We thank Nihat Ay for helpful remarks with the manuscript.

\bibliographystyle{abbrv}
\bibliography{myreferenzen}{}

\begin{thebibliography}{10}

\bibitem{Colbourn:2006:HCD:1202540}
C.~J. Colbourn and J.~H. Dinitz.
\newblock {\em Handbook of Combinatorial Designs, Second Edition (Discrete
  Mathematics and Its Applications)}.
\newblock Chapman \& Hall/CRC, 2006.

\bibitem{erdosspencer}
P.~Erd\H{o}s and J.~Spencer.
\newblock {\em Probabilistic Methods in Combinatorics}.
\newblock Academic Press Inc, 1974.

\bibitem{NIPS1991_535}
Y.~Freund and D.~Haussler.
\newblock Unsupervised learning of distributions on binary vectors using two
  layer networks.
\newblock In J.~E. Moody, S.~J. Hanson, and R.~P. Lippmann, editors, {\em
  Advances in Neural Information Processing Systems 4}, pages 912--919.
  Morgan-Kaufmann, 1992.
\newblock
  \url{http://papers.nips.cc/paper/535-unsupervised-learning-of-distributions-on-binary-vectors-using-two-layer-networks.pdf}.

\bibitem{4767596}
S.~Geman and D.~Geman.
\newblock Stochastic relaxation, {G}ibbs distributions, and the {B}ayesian
  restoration of images.
\newblock {\em Pattern Analysis and Machine Intelligence, IEEE Transactions
  on}, PAMI-6(6):721--741, 1984.
\newblock \url{http://dx.doi.org/10.1109/TPAMI.1984.4767596}.

\bibitem{octave}
S.~H. John W.~Eaton, David~Bateman and R.~Wehbring.
\newblock {\em {GNU Octave} version 4.0.0 manual: a high-level interactive
  language for numerical computations}.
\newblock 2015.
\newblock \url{http://www.gnu.org/software/octave/doc/interpreter}.

\bibitem{Kahle2010}
T.~Kahle.
\newblock Neighborliness of marginal polytopes.
\newblock {\em Beitr\"age zur Algebra und Geometrie}, 51(1):45--56, 2010.
\newblock \url{http://www.emis.de/journals/BAG/vol.51/no.1/4.html}.

\bibitem{LeRoux:2008:RPR:1374176.1374187}
N.~Le~Roux and Y.~Bengio.
\newblock Representational power of restricted {B}oltzmann machines and deep
  belief networks.
\newblock {\em Neural Computation}, 20(6):1631--1649, June 2008.
\newblock \url{http://dx.doi.org/10.1162/neco.2008.04-07-510}.

\bibitem{montufar2013mixture}
G.~Mont{\'u}far.
\newblock Mixture decompositions of exponential families using a decomposition
  of their sample spaces.
\newblock {\em Kybernetika}, 49(1):23--39, 2013.
\newblock \url{http://www.kybernetika.cz/content/2013/1/23}.

\bibitem{montufar2015deep}
G.~Mont{\'u}far.
\newblock Deep narrow {B}oltzmann machines are universal approximators.
\newblock {\em International Conference on Learning Representations 2015 (ICLR
  2015), San Diego, CA, USA}, 2015.
\newblock \url{http://www.arxiv.org/abs/1411.3784}.

\bibitem{montufar2011refinements}
G.~Mont{\'u}far and N.~Ay.
\newblock Refinements of universal approximation results for deep belief
  networks and restricted {B}oltzmann machines.
\newblock {\em Neural Computation}, 23(5):1306--1319, 2011.
\newblock \url{http://dx.doi.org/10.1162/NECO_a_00113}.

\bibitem{montufar2014expressive}
G.~Mont{{\'u}}far, N.~Ay, and K.~Ghazi-Zahedi.
\newblock Geometry and expressive power of conditional restricted {B}oltzmann
  machines.
\newblock {\em Journal of Machine Learning Research}, 16:2405--2436, 2015.
\newblock \url{http://jmlr.org/papers/v16/montufar15b.html}.

\bibitem{montufar2012does}
G.~Mont\'ufar and J.~Morton.
\newblock When does a mixture of products contain a product of mixtures?
\newblock {\em SIAM Journal on Discrete Mathematics}, 29:321--347, 2015.
\newblock \url{http://dx.doi.org/10.1137/140957081}.

\bibitem{JCD:JCD5}
K.~J. Nurmela and P.~R.~J. \"Osterg\r{a}rd.
\newblock New coverings of t-sets with (t+1)-sets.
\newblock {\em Journal of Combinatorial Designs}, 7(3):217--226, 1999.
\newblock
  \url{http://dx.doi.org/10.1002/(SICI)1520-6610(1999)7:3<217::AID-JCD5>3.0.CO;2-W}.

\bibitem{oeis}
OEIS.
\newblock The on-line encyclopedia of integer sequences, {A066010} triangle of
  covering numbers {T}(n,k) = {C}(n,k,k-1), n $>$= 2, 2 $<$= k $<$= n.
\newblock Published electronically at \url{http://oeis.org}, 2010.

\bibitem{Roedl198569}
V.~R{\"o}dl.
\newblock On a packing and covering problem.
\newblock {\em European Journal of Combinatorics}, 6(1):69--78, 1985.
\newblock \url{http://dx.doi.org/10.1016/S0195-6698(85)80023-8}.

\bibitem{e17042304}
B.~Steudel and N.~Ay.
\newblock Information-theoretic inference of common ancestors.
\newblock {\em Entropy}, 17(4):2304, 2015.
\newblock \url{http://dx.doi.org/10.3390/e17042304}.

\bibitem{Younes1996109}
L.~Younes.
\newblock Synchronous {B}oltzmann machines can be universal approximators.
\newblock {\em Applied Mathematics Letters}, 9(3):109--113, 1996.
\newblock \url{http://dx.doi.org/10.1016/0893-9659(96)00041-9}.

\bibitem{Piotr:tree}
P.~Zwiernik and J.~Q. Smith.
\newblock Tree cumulants and the geometry of binary tree models.
\newblock {\em Bernoulli}, 18:290--321, 2012.
\newblock \url{http://dx.doi.org/10.3150/10-BEJ338}.

\end{thebibliography}

\appendix

\section{Proofs}
\label{sec:proofs}

\begin{proof}[Proof of Lemma~\ref{proposition:region}]
  Let $B'=B\setminus\{m\}$.
The edge coefficients satisfy 
\begin{equation*}
	K_{B'}(w_{B'},c) = \sum_{C\subseteq B'} (-1)^{|B' \setminus C|} \log\left(1 + \exp\Big(\sum_{i\in C} w_i + c\Big)\right) 
\end{equation*}
and
\begin{equation*}
	K_B(w_B,c) = K_{B'}(w_{B'},c + w_m) - K_{B'}(w_{B'},c) .
\end{equation*} 
Using this structure, we now proceed with the proof of the individual cases. 

\smallskip\noindent\textbf{The case $|B'|=0$.} 
We omit this simple exercise. 

\smallskip\noindent\textbf{The case $|B'|=1$.} 
The \emph{if} statement is as follows. 
The elements of the set $\{0,1\}^B$ are the vertices of the $|B|$-dimensional unit cube. 
We call two vectors $x,x'\in \{0,1\}^B$ adjacent if they differ in exactly one entry, in which case they are the vertices of an edge of the cube. 

The weights $w_B$ and $c$ can be chosen such that the affine map $\{0,1\}^B\to \RR;\; x_B\mapsto w_B^\top x_B + c$ maps any chosen pair of adjacent vectors to any arbitrary values and all other vectors to large negative values. 
The soft-plus function is monotonically increasing, taking value zero at minus infinity and plus infinity at plus infinity. 
Hence, for any $s,s'\in \RR_+$, one finds weights $w$ and $c$ such that 
\begin{equation*}
\phi(x) = 
\left\{ \begin{array}{ll}
s,  & (x_{B'},x_{m}) = (1,\ldots, 1,1)  \\
s', & (x_{B'},x_{m}) = (1,\ldots, 1,0) \\
\approx 0 ,	  & \text{otherwise} 
\end{array}
\right. , 
\end{equation*}
or, alternatively, such that 
\begin{equation*}
\phi(x) = 
\left\{ \begin{array}{ll}
s,  & (x_{B'},x_{m}) = (1,\ldots, 1,0,1)  \\
s', & (x_{B'},x_{m}) = (1,\ldots, 1,0,0) \\
\approx 0 ,	  & \text{otherwise} 
\end{array}
\right. . 
\end{equation*}
This leads to $K_B \approx (s-s')$ and $K_{B'}\approx s'$ or, alternatively, $K_B \approx -(s-s')$ and $K_{B'}\approx -s'$. 
The approximation can be made arbitrarily precise.

The \emph{only if} statement is as follows.  
Denote the soft-plus function by $f\colon \RR\to \RR_+;\; s\mapsto \log(1+\exp(s))$. 
Since $|B'|=1$, $C\subseteq B'$ implies $C=B'$ or $C=\emptyset$. 
We have that $K_{B'}(w_{B'},c) = f(w_{B'} + c) - f(c)$ and $K_{B'}(w_{B'},c+w_m) = f(w_{B'} + c + w_m) - f(c + w_m)$ are either both positive or both negative, depending on the sign of~$w_{B'}$. 
If both are positive, then $K_{B}(w_{B},c) = K_{B'}(w_{B'},c+w_{m}) - K_{B'}(w_{B'},c)\ge - K_{B'}(w_{B'},c)$, and
similarly in the case that both are negative.

\smallskip\noindent\textbf{The case $|B'|=2$.}
The \emph{if} statement follows from the previous case $|B'|=1$. 
Indeed, consider an edge pair $(C,C')$ with an element more than the edge pair $(B,B')$, such that $B = C \setminus\{n\}$ and $B' = C'\setminus\{n\}$. 
Then, for any $w_B$ and~$c$, 
choosing $w_n$ large enough one obtains an arbitrarily accurate approximation 
$K_C((w_{B},w_n), c - w_n) \approx K_B(w_B,c)$ and $K_{C'}((w_{B'}, w_n), c - w_n) \approx K_{B'}(w_{B'},c)$. 

For the \emph{only if} statement we use a similar argument as previously. 
We have $K_{B'}(w_{B'}, c) = f(w_{1}+w_2+c) + f(c) - f(c + w_1) - f(c + w_2)$. 
By convexity of~$f$, this is non-negative if and only if either $w_{1},w_{2}\ge0$ or $w_{1},w_{2}\le 0$.
In other words, this is non-negative if and only if $w_1\cdot w_2\ge 0$. 
Under either of these conditions, $K_{B'}(w_{B'}, c + w_m)$ is also non-negative. 
Similarly, $K_{B'}(w_{B'}, c)$ is non-positive if and only if 
$w_{1}\cdot w_{2}\le 0$.
In this case, $K_{B'}(w_{B'}, c + w_m)$ is also non-positive. 
Now the statement follows as in the case $|B'|=1$. 

\smallskip\noindent\textbf{The case $|B'|=3$.} 
We need to show that any edge pair coefficients can be represented. 
Consider first $J_{B'}\geq 0$. 
We choose weights of the form $w_{B'} = \omega 1_{B'}$, where $\omega\in\mathbb{R}$ and $1_{B'}$ is the vector of $|B'|$ ones. 
% with a large $\omega\in\RR$ and an appropriate $c$. 
Then $K_{B'}(w_{B'},c) = f(3 \omega + c) - 3 f(2 \omega + c) + 3 f(\omega + c) - f(c)$. 
We can choose $\omega$ and $c$ such that $3\omega +c = f^{-1}(J_{B'})$ while $2\omega +c, \omega c, c$ take very large negative values. 
This yields $K_{B'}\approx J_{B'}$. 

Note that the derivative of the soft-plus function is the logistic function, i.e., $f'(s) = 1/(1+\exp(-s))$. 
Choosing $\omega$ large enough from the beginning, 
the function $w_m\mapsto K_{B'}(w_{B'}, c + w_m)$ is monotonically increasing in the interval $w_m\in [0, \omega/2]$ and surpasses the value $\frac{1}{5}\omega$. 
On the other hand, when $w_m$ is large enough, depending on $\omega$ and $c$, 
we have that $2\omega + c + w_m \geq \frac{5}{12} (3\omega + c + w_m)$ and $f(2\omega + c + w_m) \geq \frac{5}{12} f (3\omega + c + w_m)$. 
In this case $f(3\omega + c + w_m) - 3 f(2\omega + c + w_m)\leq -\frac{1}{4}(3\omega + c + w_m)\leq -\frac{1}{4}\omega$. 
At the same time, $\omega + c + w_m$ and $c+ w_m$ are smaller than $-\frac{1}{12}\omega$ and so $f(\omega + c + w_m)$ and $f(c+ w_m)$ are very small in absolute value. 

By the mean value theorem, depending on $w_m$, $K_{B'}(w_{B'}, c + w_m)$ takes any value in the interval $[-\frac15\omega, \frac15\omega]$, where $\omega$ is arbitrarily large. 
In turn, we can obtain $K_B = K_{B'}(w_B',c + w_m) - K_{B'}(w_{B'},c) \approx J_{B}$ for any $J_{B}\in \RR$. 

For $J_{B'}\leq 0$ the proof is analogous after label switching for one variable. 

\smallskip\noindent\textbf{The case $|B'|>3$.} 
This follows from the previous case $|B'|=3$ in the same way that the \emph{if} part of the case $|B'|=2$ follows from the case $|B'|=1$. 
\end{proof}

\end{document}